\documentclass[12pt]{amsart}
\usepackage{amsfonts,amssymb,amscd,amsmath,enumerate,verbatim,calc,amsthm}
\input xy
\xyoption{all}

\headsep=1cm \numberwithin{equation}{section}
\newtheorem{thm}{Theorem}[section]
\newtheorem{cor}[thm]{Corollary}
\newtheorem{con}[thm]{Convention}
\newtheorem{lem}[thm]{Lemma}
\newtheorem{prop}[thm]{Proposition}
\newtheorem{defn}[thm]{Definition}

\newtheorem{exam}[thm]{Example}

\newcommand{\Ann}{\operatorname{Ann}\,}
\newcommand{\coker}{\operatorname{Coker}\,}
\newcommand{\Hom}{\operatorname{Hom}\,}

\newcommand{\Spec}{\operatorname{Spec}\,}
\newcommand{\Max}{\operatorname{Max}\,}

\newcommand{\Ass}{\operatorname{Ass}\,}
\newcommand{\Assh}{\operatorname{Assh}\,}
\newcommand{\Att}{\operatorname{Att}\,}
\newcommand{\Supp}{\operatorname{Supp}\,}

\newcommand{\grad}{\operatorname{grade}\,}
\newcommand{\depth}{\operatorname{depth}\,}
\renewcommand{\dim}{\operatorname{dim}\,}

\newcommand{\cd}{\operatorname{cd}\,}

\newcommand{\Tr}{\operatorname{Tr}\,}

\newcommand{\Min}{\operatorname{Min}\,}

\newcommand{\h}{\operatorname{ht}\,}

\newcommand{\fa}{\mathfrak{a}}
\newcommand{\fb}{\mathfrak{b}}

\newcommand{\fm}{\mathfrak{m}}
\newcommand{\fp}{\mathfrak{p}}
\newcommand{\fq}{\mathfrak{q}}

\newcommand{\fx}{\mathfrak{x}}

\begin{document}
\bibliographystyle{amsplain}


\title[Characterization of some special rings via Linkage]
 {Characterization of some special rings via Linkage}

\bibliographystyle{amsplain}

     \author[M. jahangiri]{Maryam jahangiri$^1$}
     \author[kh. sayyari]{khadijeh sayyari$^2$}

\address{$^{1, 2}$ Faculty of Mathematical Sciences and Computer,
Kharazmi University, Tehran, Iran.}




\begin{abstract}
Some descriptions of linked ideals in a commutative Notherian ring $R$ are provided in terms of the Associated prime ideals of $R.$ Then, among other things, we make some characterization of Cohen-Macaulay, Gorenstein and regular local rings in terms of their linked ideals.
\end{abstract}

\maketitle

\bibliographystyle{amsplain}
\section{introduction}

 The theory of linkage is an important topic in commutative algebra and Algebraic Geometry. It refers to Halphen (1870) and M. Noether \cite{No}(1882) who worked to classify space curves. In 1974 the significant work of Peskine and Szpiro \cite{PS} stated this theory in the modern algebraic language; two proper ideals $\fa$ and $\fb$ in a ring $R$ is said to be linked if there is an regular sequence $\underline{\fx}$ in their intersection such that $\fa = (\underline{\fx}) :_R \fb$ and $\fb = (\underline{\fx}) :_R \fa$.

In a resent paper \cite{JS}, inspired by the works in the ideal case, the authors present the concept of the linkage of ideals with respect to a module. Let $R$ be a commutative Noetherian ring with $1\neq 0$ and $M$ be a finitely generated $R$-module.
Let $\fa$, $\fb$ and $I$ be ideals of $R$ with  $I\subseteq \fa \cap \fb$  such that $I$ is generated by an $M$-regular sequence, $M\neq \fa M$ and $M\neq \fb M$. Then, $\fa$ and $\fb$ are said to be linked by   $I$ with respect to $M$ if $\fb M = IM:_M\fa$ and $\fa M = IM:_M\fb $. This is a generalization of the classical concept of linkage when $M= R$.

One of the main problems in this subject is to determine when the ideal $\fa$ of $R$ is a linked ideal. In this paper, first we consider the above generalization and define the set $S_{(I;M)}$, which contains the set of linked radical ideals of $R$ with respect to $M$ by $I.$ In Section 2, we study some of the basic properties of this set and, using them, we characterize the radical linked ideals. Indeed, among other things, we demonstrate that the radical ideal $\fa$ of $R$ is a linked ideal if and only if there exists an $R$-regular sequence $\underline{\fx}$ of length $\grad \fa$ in $\fa$ and $\Lambda \subseteq \Ass \frac{R}{\fx}$ such that $\fa = \bigcap_{\fp \in \Lambda} \fp$ (Corollary \ref{t6}). Then, we show that, in a Cohen-Macaulay local ring $R$, the radical ideal $\fa$ is a linked ideal if and only if $\fa$ is unmixed (Corollary \ref{c3}).

In the theory of local cohomology modules one of the main problems is the computing of the annihilator of these modules. As an application of linkage theory, we study the annihilator of local cohomology modules in some special cases, too (Proposition \ref{p1} and Example \ref{ex2}).

In Section 3, using the results provided in Section 2, we characterize Cohen-Macaulay, Gorenstein and regular local rings in terms of their linked ideals (see theorems \ref{t10}, \ref{t13} and \ref{t12}).

Throughout the paper, $R$ denotes a non-trivial commutative Noetherian ring, $\fa$ and $\fb$ are non-zero proper ideals of $R$ and $M$ will denote a finitely generated $R$-module.


\section{linked ideals with respect to a module}


In this section, first, we study some basic properties of the linked ideals with respect to a module and provide some characterization of them. Then, using this characterization, we make a description of linked ideals in $R.$

\begin{defn}\label{F1}
 \emph{Assume that $\fa M\neq M$ and $\fb M\neq M$ and let $I\subseteq \fa \cap \fb$ be an ideal which is generated by an $M$-regular sequence. Then we say that the ideals $\fa$ and $\fb$ are linked by $I$ with respect to $M$, denoted by $\fa\sim_{(I;M)}\fb$, if $\fb M = IM:_M\fa$ and $\fa M = IM:_M\fb $. Also, the ideals $\fa$ and $\fb$ are said to be geometrically linked by $I$ with respect to $M$ if $\fa M \cap \fb M = IM$. The ideal $\fa$ is $M$-selflinked by $I$ if $\fa\sim_{(I;M)}\fa$.}
\end{defn}

Note that in the case where $M = R$, this concept is the classical concept of linkage of ideals in \cite{PS}.

In the following, we contract the set of ideals $S_{(I;M)}$ which plays a fundamental role in the paper. This set, as we show, also contains of all linked radical ideals by $I$ with respect to $M$.

\begin{con}\label{F2}
 Assume that $I$ is an ideal of $R$ which is generated by an $M$-regular sequence. Set $$ S_{(I;M)} : = \{ \fa \triangleleft R | I \subsetneqq \fa,  \fa = IM:_R IM:_M \fa  \}.$$

 Note that $S_{(I;R)}$ actually contains of all linked ideals by $I$.
\end{con}

 Some basic properties of the set $S_{(I;M)}$ are presented in the following Lemma.
\begin{lem}\label{l13}
Let $I$ be an ideal of $R$ which is generated by an $M$-regular sequence. Then
\begin{itemize}
   \item [(i)]  $\Ass \frac{R}{\fa}\subseteq \Ass \frac{M}{IM}$, for all $\fa \in S_{(I;M)}$.
   \item [(ii)]  $\Ass \frac{M}{IM}-\{I\} =\Spec S_{(I;M)}$, the set of prime ideals of $S_{(I;M)}$.
   \item [(iii)] $S_{(I;M)}$ is close under finite intersection. More precisely, $\fa_1 \cap \fa_2 \in S_{(I;M)}$, for all $\fa_1,\fa_2 \in S_{(I;M)}$ with $\fa_1 \cap \fa_2 \neq I$.
   \item [(iv)]  $S_{(I;M)}$ is close under radical. In particular, if $\fa \in S_{(I;M)}$ then $\sqrt{\fa}= \bigcap_{\fp \in \Lambda} \fp\in S_{(I;M)}$ for some $\Lambda \subseteq \Ass \frac{M}{IM}- \{I\}$.
   \item [(v)]  $\sqrt{\fa + \Ann M} \in S_{(I;M)}$, for all ideals $\fa$ of $R$ which are linked by $I$ with respect to $M$.
   \end{itemize}
\end{lem}

\begin{proof}
Note that it is enough to consider the case where $I =0$.

\begin{itemize}
   \item [(i)]  Let $\fa \in S_{(0;M)}$, $N:= 0:_M \fa $ and assume that $N =  \Sigma^t_{i=1}R\alpha_i$ for some $\alpha_1,...,\alpha_t\in N.$ Then, by the assumption, $\fa= 0:_RN$ and using the natural monomorphism $\frac{R}{0:_R N}\rightarrow \oplus^t_{i=1} R\alpha_i$, we get $$\Ass \frac{R}{\fa} \subseteq \bigcup^t_{i=1} \Ass R\alpha_i \subseteq \Ass M.$$
   \item [(ii)]  Let $\alpha \in M$ such that $\fp = 0:_R \alpha \in \Ass M-\{0\}$. Hence, $$\fp = 0:_R\alpha = 0:_R 0:_M \fp\in S_{(I;M)}.$$ The converse follows from (i).
   \item [(iii)]  It follows from the fact that $$\fa_1 \cap \fa_2 \subseteq 0:_R 0:_M (\fa_1 \cap \fa_2) \subseteq (0:_R 0:_M \fa_1)\cap (0:_R 0:_M \fa_2).$$
   \item [(iv)]  First note that $\sqrt{\fa}$ is a non-zero ideal. Hence $\sqrt{\fa} = \bigcap_{\fp \in \Min \fa} \fp \in S_{(0;M)}$, by (iii) and (ii).
   \item [(v)]  By \cite[3.7(ii)]{JS}, $\sqrt{\fa + \Ann M} = \sqrt{0:_R 0:_M \fa}$. Therefore $$\Min (\fa + \Ann M) = \Min \Ass (0:_M \fa) \subseteq \Ass M.$$ Hence $\sqrt{\fa + \Ann M} = \bigcap_{\fp \in \Min (\fa + \Ann M)} \fp \in S_{(0;M)},$ by (iii) and (ii).
   \end{itemize}

\end{proof}

In the following proposition we study when the set $S_{(I;M)}$ is empty.
\begin{prop} \label{l6}
Let $I$ be a proper ideal of $R$ which is generated by an $M$-regular sequence. Then, the following statements hold.
\begin{itemize}
   \item [(i)] If $S_{(I;M)}= \emptyset$ then $I$ is prime.
   \item [(ii)]  If $I$ is prime and $M$ is flat or $\frac{M}{IM}$ is torsion-free then $S_{(I;M)}= \emptyset$.
   \item [(iii)]  If $S_{(I;M)} \neq \emptyset$ then $\Max S_{(I;M)} = \Max \Ass \frac{M}{IM}.$
   \end{itemize}

\end{prop}

\begin{proof}
 \begin{itemize}
   \item [(i)] Note that $\Ass \frac{M}{IM} \neq \emptyset$ and so, by the assumption and \ref{l13}(ii), $I \in \Ass \frac{M}{IM}$.
   \item [(ii)] If $M$ is flat then $\Ass \frac{M}{IM} = \{I\}$, by \cite[23.2]{M}.

   Also, if $\frac{M}{IM}$ is torsion-free then, as $\frac{ M}{IM}$ has rank, it embeds in a finite copies of $\frac{R}{I}$. Hence $\Ass\frac{M}{IM} \subseteq \Ass \frac{R}{I}= \{I\}$.

        Now, assume that there exists $\fa \in S_{(I;M)}$. Then, by \ref{l13}, $$\sqrt{\fa} = \bigcap_{\fp \in \Lambda}\fp \in S_{(I;M)}$$ for a subset $\Lambda$ of $\Ass \frac{M}{IM}-\{I\} = \emptyset$, and this is a contradiction.

   \item [(iii)] Let $\fa \in \Max S_{(I;M)}$. Then, by \ref{l13}(iv), $\fa= \sqrt{\fa }= \bigcap_{\fp \in \Lambda} \fp\in S_{(I;M)}$ for some $\Lambda \subseteq \Ass \frac{M}{IM}-\{I\}$.  Therefore, by \ref{l13}(ii), $\fa \in \Max \Ass\frac{M}{IM}$.

       Now, let $\fp \in \Max \Ass\frac{M}{IM}$. Then, $\fp \neq I$, otherwise, as $$\Max S_{(I;M)} \subseteq \Max \Ass \frac{M}{IM} \subseteq V(I),$$ $\fp =I \subseteq \fq$ for all $\fq \in \Max S_{(I;M)}$. So, $I= \fq\in S_{(I;M)}$, which is a contradiction.

       Therefore, $\fp \in \Ass \frac{M}{IM} - \{I\}$ and, by \ref{l13}(ii), $\fp \in \Max S_{(I;M)}$.
   \end{itemize}

\end{proof}

As a corollary of the above proposition, we have a criterion for the existence of an ideal of $R$ linked by an $R$-regular sequence $I$.
\begin{cor} \label{c11}
Let $I$ be a proper ideal of $R$ which is generated by an $R$-regular sequence. Then, the following statements hold.
\begin{itemize}
   \item [(i)] There is no ideal of $R$ linked by $I$ if and only if $I$ is a prime ideal.
   \item [(ii)]  If there is an ideal of $R$ linked by $I$ then $$\Max  \{ \fa \triangleleft R | \fa \textrm{ is linked by } I \} = \Max \Ass \frac{R}{I}.$$
   \end{itemize}

\end{cor}

We can also characterize linked ideals, as follows.
\begin{cor}\label{t6}
The following statements hold.
\begin{itemize}
   \item [(i)]  A radical ideal $\fa$ is linked if and only if $\fa = \bigcap_{\fp \in \Lambda} \fp$ for some $R$-regular sequence $(\fx)$ and some $\Lambda \subseteq \Ass \frac{R}{\fx}$.
   \item [(ii)]  Non-zero maximal ideals of $R$ are linked ideals. Moreover, the non-zero elements of $\Min R$ are linked ideals.
   \item [(iii)] If $R$ isn't reduced then the nilradical ideal $\sqrt{0}$ is a linked ideal.
   \end{itemize}
\end{cor}

\begin{proof}
(i) Assume that $\fa$ is a linked ideal by $I$. Hence, by \cite[Proposition 5. p594]{MS}, $\Ass \frac{R}{\fa} \subseteq \Ass \frac{R}{I}$ and $\fa = \bigcap_{\fp \in \Lambda} \fp$ for some $\Lambda \subseteq \Ass \frac{R}{I}$.

Now, let $\underline{\fx} = (x_1,...,x_n)$ be an regular sequence and $\fa = \bigcap_{\fp \in \Lambda} \fp$ for some $\Lambda \subseteq \Ass \frac{R}{\fx}$. Then, in view of \cite[6.3 and Exersice 6.7]{M} and considering the regular sequence $(x_1,...,x_{n-1}, x^2_n)$, we may assume that $\underline{\fx} \notin \Ass \frac{R}{\fx}$ and that $\fa \supsetneqq \underline{\fx}$. Now, the result follows from \ref{l13}(ii) and (iii).

(ii) and (iii) follow from (i).

\end{proof}
One may ask whether linking over a module implies linking over the ring and vice versa. In the following corollary, we consider a case where linking over the canonical module implies linking over $R$. For some other cases, we refer the reader to \cite[\S4]{JS}.

\begin{cor}\label{l8}
Let $(R, \fm)$ be an unmixed complete local ring with the canonical module $w_R$ and $\fa$ and $\fb$ be two ideals of $R$ such that $\fa\sim_{(0;w_R)}\fb$. Then $\sqrt{\fa}$ is a linked ideal over $R$.
\end{cor}

\begin{proof}
 By local duality theorem \cite[11.2.6]{BS} and \cite[11.2.7(iii)]{BS}, $$w_R \cong \Hom_R(H^{\dim R}_{\fm}(R), E(\frac{R}{\fm})).$$ Therefore, in view of  \cite[11.3.6 and 10.2.20]{BS}, \begin{equation}\label{e4}
                                                                              \Ass w_R =\Att H^{\dim R}_{\fm}(R) = \Assh R.
                                                                            \end{equation}
 Also, using \cite[2.2(e)]{HH}, $\Ann w_R = Ker(R \rightarrow \Hom_R(w_R,w_R)) = 0$. Hence, by \ref{l13}, $\sqrt{\fa} = \bigcap_{\fp \in \Lambda} \fp$ for some $\Lambda \subseteq \Ass w_R$. Now, the result follows from (\ref{e4}) and \ref{t6}(i).

\end{proof}

\begin{defn}\label{F3}
 \emph{Let $(-)^*:= \Hom_R(-,R)$ and consider an exact sequence $F_2 \overset{f}{\rightarrow} F_1 \overset{g}{\rightarrow} M \rightarrow 0$, where $F_1$ and $F_2$ are free $R$-modules. Setting $\Tr M := \coker f^*$ and $\lambda M := \Omega \Tr M$, where "$\Omega$" is the first syzygy module, we get the exact sequences $$0 \rightarrow M^* \overset{g^*}{\rightarrow} (F_1)^* \rightarrow \lambda M \rightarrow 0$$ and $$0 \rightarrow M^* \overset{g^*}{\rightarrow} (F_1)^* \overset{f^*}{\rightarrow} (F_2)^* \rightarrow \Tr M \rightarrow 0.$$ Now, following \cite{MS}, finitely generated $R$-modules $M$ and $N$ are said to be horizontally linked, denoted by $M \sim N$, if $ \lambda M \cong N$ and $\lambda N \cong M$.}
 \end{defn}

As another example of linked ideals we have the following corollary. Also, in \cite{DST}, the authors study the relation between linkness of $M$ and that of $\Ann M$ as an ideal. In the following, we consider this problem, too.
\begin{cor}\label{c1}
Let $M$ be a horizontally linked $R$-module such that $\Ann M \neq 0$ or $R$ is not reduced. Then $\sqrt{\Ann M }$ is a linked ideal.
\end{cor}

\begin{proof}

  As $M$ is a syzygy, $\Ass M \subseteq \Ass R$ and $\sqrt{\Ann M } = \cap_{\fp \in \Lambda} \fp$ for some $\Lambda \subseteq \Ass R$. Also, using the assumption, $0 \subsetneqq \sqrt{0} \subseteq \sqrt{\Ann M }$ and the result follows from \ref{t6}.
\end{proof}

In the next two items we describe the ideals that are linked by a radical ideal and show that they are, actually, geometrically linked.

\begin{thm}\label{l4}
Let $I$ be an ideal of $R$ which is generated by an $R$-regular sequence and $\Ass \frac{R}{I} = \Min \Ass \frac{R}{I}.$ Let $I = \cap_{i=1}^n \fq_i$ be the minimal primary decomposition of $I$. Then
\begin{itemize}
\item[(i)]  $\cap_{i\in\Lambda}\fq_i$ and $\cap_{i\in\{1,...,n\}-\Lambda}\fq_i$ are geometrically linked by $I$, where $\Lambda\subset \{1,...,n\}.$
\item[(ii)]  If $I$ is radical then all ideals which are linked by $I$ are radical.
\end{itemize}
\end{thm}

\begin{proof}
\begin{itemize}
\item[(i)] Let $\fq_i$ be $\fp_i$-primary, for all $i=1,...,n$. We have
 \begin{eqnarray*}
 I:(\cap_{i=1}^r\fq_i)  &=& \cap_{j=1}^n(\fq_j: (\cap_{i=1}^r\fq_i))\\
 & =& \cap_{j=r+1}^n(\fq_j: (\cap_{i=1}^r\fq_i))= \cap_{j=r+1}^n\fq_j
 \end{eqnarray*}
  The last equality follows from the fact that $\cap_{i=1}^r\fq_i \nsubseteq \fp_j$ for all $j>r$.
\item[(ii)]
Assume that $\fa$ is linked by $I$. If $\fa = \cap_{j=1}^r Q_{j}$ is a minimal primary decomposition of $\fa$ then, by \cite[Proposition 5. p594]{MS}, for all $j = 1,...,r,$ $Q_{j}$ is $\fp_{j}$-primary for some $\fp_j \in \Ass \frac{R}{I}$. For all $j = 1,...,r,$ we have $$I \subseteq Q_j \cap \cap_{\fp \in \Ass \frac{R}{I}-\{\fp_j\}} \fp \subseteq \cap_{\fp \in \Ass \frac{R}{I}}\fp = I.$$ Therefore, $Q_j \cap \cap_{\fp \in \Ass \frac{R}{I}-\{\fp_j\}} \fp $ is another minimal decomposition of $I$. Via of $\Ass \frac{R}{I} = \Min \Ass \frac{R}{I}$ and second uniqueness theorem, $Q_j = \fp_j$ for all $j = 1,...,r$. Therefore, $\fa$ is radical.
\end{itemize}

\end{proof}

\begin{cor}

Let $I$ be a radical ideal of $R$ which is generated by an $R$-regular sequence. Then the ideal $\fa$ is linked by $I$ if and only if $\fa = \cap _{\fp \in \Lambda} \fp$ for some $\Lambda \subset\Ass \frac{R}{I}$. In this case, $\fa$ and $\fb:= \cap _{\fp \in \Ass \frac{R}{I}- \Lambda} \fp$ are geometrically linked.
\end{cor}

In the theory of local cohomology modules, computing the annihilator of these modules attracts lots of interest, see for example \cite{BRS}, \cite{SS} and \cite{Z}.

The following proposition consider a case where the annihilator of some local cohomology modules are linked. For another case see example \ref{ex2}.

\begin{prop}\label{p1}
Let $(R, \fm)$ be a complete local ring of dimension $d>0$ and $\fa$ and $\fb$ be two ideals of $R$ such that $\fa\sim_{(0;R)}\fb$ and $\cd(\fa,R)=d$. Then, the following statements hold.
 \begin{itemize}
   \item [(i)]  $\sqrt{0:H^d_{\fa}(R)}$ is a linked ideal.
   \item [(ii)]  If $R$ is unmixed and $\fa+\fb$ is $\fm$-primary then $0:H^d_{\fa}(R)=\fb$ and $0:H^d_{\fb}(R)=\fa.$
  \end{itemize}
  \end{prop}

\begin{proof}
\begin{itemize}
   \item [(i)]   First we claim that $0:H^d_{\fa}(R)\neq 0$. Suppose the contrary. Then, in view of \cite[2.4]{L}, $\Ass R= \Assh R$ and $\sqrt{\fa +\fp}=\fm$ for all $\fp \in \Assh R.$ On the other hand, by \cite[Proposition 5. p594]{MS}, there are some $\fp\in\Ass R$ such that $\fp \supseteq \fa.$ This implies that $\fp = \fm$ which is a contradiction.

 Now, let $0= \cap_{i=1}^n \fq_i$ be a minimal primary decomposition of 0 such that $\fq_i$ is $\fp_i$-primary, for all $i=1,...,n$. Then, by \cite[2.4]{L}, \begin{equation}\label{e} 0:H^d_{\fa}(R) = \cap_{j=1}^r \fq_{i_j},\end{equation} for some $\{i_1,...,i_r\}\subset\{1,...,n\}.$ Therefore, by \ref{t6}, $\sqrt{0:H^d_{\fa}(R)}$ is a linked ideal.

\item [(ii)]   Let $R$ be unmixed. Hence, by theorem \ref{l4} and (\ref{e}), $0:H^d_{\fa}(R)$ is a linked ideal. Also, via \cite[8.2.6]{BS} and the fact that $d>0$, $\Att H^d_{\fa}(R) \subseteq \Ass R - V(\fa).$ Moreover, let $\fp \in\Ass R - V(\fa)$. Then, in view of \cite[2.2]{JS}, $\fp \supseteq \fb$ and by the assumption $\sqrt{\fa +\fp}=\fm$. This implies that $\fp \in \Att H^d_{\fa}(R).$ Therefore, $\Att H^d_{\fa}(R) = \Ass R - V(\fa)$ and \begin{equation}\label{e1} 0:H^d_{\fa}(R) = \bigcap ^n_{i= 1, \fp_i\nsupseteq \fa} \fq_i. \end{equation}

 Similarly, $\Att H^d_{\fb}(R) = \Ass R - V(\fb).$ We claim that $H^d_{\fb}(R)\neq 0$. Suppose the contrary, i.e. $\Ass R = V(\fb)$. So, by \cite[Proposition 5. p594]{MS}, there are some $\fp\in\Ass R$ such that $\fp \supseteq \fa + \fb.$ It follows from the assumption that $\fp= \fm$ which is a contradiction. Then, \begin{equation}\label{e2}0:H^d_{\fb}(R) = \bigcap^{n}_{i= 1, \fp_i\nsupseteq \fb} \fq_i. \end{equation}

On the other hand, let $\fa = \cap_{i=1}^k Q_{i}$ and $\fb = \cap_{j=1}^l Q'_{j}$ be the minimal primary decompositions of $\fa$ and $\fb.$ Then, via the fact that $\Ass R \cap V(\fa + \fb)= \emptyset,$ $\fa$ and $\fb$ are geometrically linked and so $\fa\cap\fb=0$. Hence, $0=\cap_{i=1}^r Q_{i}\bigcap \cap_{j=1}^l Q'_{j}$ is another minimal primary decompositions of 0 and using the second uniqueness theorem, without lose of generality, one may assume that $\fa=\cap_{i=1}^r\fq_i$ and $\fb = \cap_{i=r+1}^n\fq_i.$ Now, let $\fp_i\nsupseteq \fa$, for some $i=1,...,n.$ Then, $\fq_i\nsupseteq \fa$ and so $i>r$ and $\fq_i\supseteq \fb.$ Also, if $\fq_i\supseteq \fb,$ for some $i=1,...,n,$ then $\fq_i\nsupseteq \fa,$ else $\fp_i\supseteq \fa$ and $\fp_i\in \Ass R \cap V(\fa + \fb)=\emptyset$. Hence $$\{\fq_i| \fq_i\supseteq \fb\} =\{\fq_i| \fq_i\nsupseteq \fa\}.$$ This implies that $\fb= \cap^{n}_{i= 1, \fp_i\nsupseteq \fa}  \fq_i$ and $\fa = \cap^{n}_{i= 1, \fp_i\nsupseteq \fb}  \fq_i$. Now, the result follows from (\ref{e1}) and (\ref{e2}).
\end{itemize}

\end{proof}


\section{characterization of some special rings in terms of linkage}

In this section, we characterize Cohen-Macaulay, Gorenstein and regular local rings in terms of the linked ideals.

\begin{prop}\label{t11}
Let $R$ be a Cohen-Macaulay ring. Then
\begin{itemize}
   \item [(i)]  $\fp$ is a linked ideal, for all $\fp \in \Spec R- \{0\}$.
   \item [(ii)]  $\fp_1 \cap \fp_2$ is a linked ideal, for all $\fp_1,\fp_2 \in \Spec R$ with $\fp_1 \cap \fp_2\neq 0$ and $\h \fp_1 = \h \fp_2$.
  \end{itemize}
\end{prop}

\begin{proof}
\begin{itemize}
   \item [(i)]  Let $\fp \in \Spec R- \{0\}$ and $t:= \h \fp$. Then there exists an $R$-regular sequence $x_1,...,x_t$ in $\fp$ such that $\fp \subseteq Z_R (\frac{R}{(x_1,...,x_t)})$. Also, there exists $\fq \in \Ass_R(\frac{R}{(x_1,...,x_t)})$ such that $\fp \subseteq \fq$. Via $\h \fp = \h \fq$, $\fp = \fq$ and, by \ref{t6}, $\fp$ is a linked ideal.
   \item [(ii)]   Let $\fp_1,\fp_2 \in \Spec R$ such that $\fp_1 \cap \fp_2\neq 0$ and $t:=\h \fp_1 = \h \fp_2$. Then there exists an $R$-regular sequence $x_1,...,x_t \in \fp_1 \cap \fp_2$. Via the proof of (i), $\fp_1, \fp_2 \in \Ass_R (\frac{R}{(x_1,...,x_t)})$ and the assertion follows, again, from \ref{t6}.
  \end{itemize}

\end{proof}
As a corollary of the above proposition and \ref{t6}(i), one can characterize the radical linked ideals in a Cohen-Macaulay ring. Recall that the ideal $\fa$ is said to be relative Cohen-Macaulay with respect to $M$ if $H^i_{\fa}(M) = 0$ for all $i \neq \grad_M\fa.$ If $M=R$, for abbreviation, $\fa$ is called relative Cohen-Macaulay.

\begin{cor}\label{c3}
Let $R$ be a Cohen-Macaulay ring. Then
\begin{itemize}
   \item [(i)]  A radical ideal $\fa$ is a linked ideal in $R$ if and only if $\fa$ is unmixed.
   \item [(ii)]  The relative Cohen-Macaulay ideals are linked.
  \end{itemize}
\end{cor}

In theorem \ref{t10}, we will show that, in a certain case, part (i) of the above corollary characterize Cohen-Macaulay rings.

\begin{exam}\label{ex2}
Let $(R, \fm)$ be a Cohen-Macaulay complete local ring and $M$ be a finitely generated $R$-module. Then, by \ref{t11}(i), every non-zero prime ideal of $\Supp M$ is a linked ideal.  Also, by \cite[7.2.11(ii) and 7.3.2]{BS}, $$\sqrt{\Ann H^{\dim M}_{\fm}(M)} = \bigcap_{\fp \in\Att H^{\dim M}_{\fm}(M)}\fp =  \bigcap_{\fp \in\Assh M}\fp$$ is an unmixed ideal. Therefore, if $\Ann H^{\dim M}_{\fm}(M) \neq 0$ then, by \ref{c3}(i), $\sqrt{\Ann H^{\dim M}_{\fm}(M)}$ is a linked ideal.
\end{exam}

In spite of the proposition \ref{t11}, there are non-Cohen-Macaulay rings for which every prime ideal is linked.

\begin{exam}\label{ex1}
Let $R$ be a one dimensional ring with $\depth R = 0$. Then $\Spec R = \Max R \cup \Min R$ and, by \ref{t6}(ii) and (iii), every prime ideal of $R$ is linked.
\end{exam}

In the rest of this section, we classify regular, Gorenstein and Cohen-Macaulay rings in terms of their linked ideals.

\begin{thm}\label{t10}
Let $(R,\fm)$ be a local ring and $\Ass M = \Min \Ass M$. Then the following statements are equivalent.
\begin{itemize}
   \item [(i)]  $M$ is Cohen-Macaulay.
   \item [(ii)] $\frac{M}{\fa M}$ is an unmixed module for all ideals $\fa$ which are linked with respect to $M$.
  \end{itemize}

In particular, $R$, with $\Ass R = \Min R$, is Cohen-Macaulay if and only if the ideal $\fa$ is unmixed provided it is a linked ideal.
\end{thm}

\begin{proof}

   $"(i) \Rightarrow (ii)"$  Let $\fa$ be an ideal which is linked by the ideal $I$ generating by an $M$-regular sequence with respect to $M$ and let $\fp \in \Ass_R \frac{M}{\fa M}$. Then, by \cite[2.7]{JS}, $\fp \in \Ass_R \frac{M}{I M}$. Via Cohen-Macaulayness of $\frac{M}{IM}$, $\dim \frac{R}{\fp} = \dim \frac{M}{IM}$. On the other hand $\dim \frac{M}{IM} \geq \dim \frac{M}{\fa M}$. Putting together both of the estimates, the desired equality is shown to be true.

   $"(ii) \Rightarrow (i)" $ In the case where $\fm\in \Ass M,$ clearly, $M$ is Cohen-Macaulay. So, assume that $\depth M>0.$ Let $t\in \mathbb{N}$, $x_1,...,x_t$ be an $M$-regular sequence and $\fp \in \Ass_R (\frac{M}{(x_1,...,x_t)M})$. By \cite[2.2]{JS}, $(x_1,...,x_t)$ is an $M$-self linked ideal and so, by the assumption, $\frac{M}{(x_1,...,x_t)M}$ is unmixed. Therefore, $$\dim M - t \geq \dim M - \h_M\fp \geq \dim \frac{R}{\fp} = \dim \frac{M}{(x_1,...,x_t)M} = \dim M - t.$$ This implies that $\h_M\fp =t$ and, hence, $M$ is Cohen-Macaulay.

\end{proof}

In \cite[2.2]{DKH} a characterization of Gorenstein local rings is presented in terms of the "generically linked" ideals, provided $R$ is a Cohen-Macaulay ring.

In the following, we have a general characterization without the assumption that $R$ is Cohen-Macaulay.

\begin{thm}\label{t13}
Let $(R,\fm)$ be a local ring. Then the following are equivalent.
\begin{itemize}
   \item [(i)]  $R$ is Gorenstein.
   \item [(ii)] Any unmixed ideal $\fa$ is linked by every $R$-regular sequence $(x_1,...,x_t) \subset\fa$ of length $t=\grad \fa$ .
  \end{itemize}
\end{thm}

\begin{proof}

   $"(i) \Rightarrow (ii)"$  Using the fact that $\frac{R}{(x_1,...,x_t)}$ is Gorenstein for every $R$-regular sequence $x_1,...,x_t$, one may assume that $\fa$ is an unmixed non-zero ideal of grade $0$. Now, the assertion follows from \cite[4.1]{JS}.

   $"(ii) \Rightarrow (i)"$ We proceed by induction on $d:= \dim R$. Let $d=0$. Then every non-zero ideal $\fa$ of $R$ is unmixed of zero grade and, by the assumption, $0:_R0:_R\fa = \fa$. Therefore, in view of \cite[3.2.15]{BH}, $R$ is Gorenstein.

    Now assume that $d>0$ and the assertion has been proved for all local ring of dimension $<d$. We claim that $\depth R>0$. Assume to the contrary that $\fm \in \Ass R$. Then, by the assumption, $\fm^j$ is linked by the zero ideal and $0:_R0:_R\fm^j = \fm^j$ for all $j\in\mathbb{N}.$ On the other hand, there is $i \in \mathbb{N}$ such that $0:_R\fm^i = 0:_R\fm^{i+1}$. This implies that $\fm^i=0$ and $d=0$, which is a contradiction.

     Now, let $x \in \fm - Z(R)$, $\overline{\fa}$ be an unmixed ideal of grade $l$ and $\overline{y_1},..., \overline{y_l}$ be an arbitrary $\overline{R}$-regular sequence in $\overline{\fa}$ such that $\overline{\fa} \neq (\overline{y_1},..., \overline{y_l})$, where $-: R \rightarrow \frac{R}{Rx}$ is the natural homomorphism. Then, by \cite[Exersice 6.7]{M}, $\fa$ is an unmixed ideal of grade $l+1$ and, by the assumption, $$(x,y_1,...,y_l):_R(x,y_1,...,y_l):_R\fa = \fa.$$ In other words, $$(\overline{y_1},..., \overline{y_l}):_{\overline{R}}(\overline{y_1},..., \overline{y_l}):_{\overline{R}}\overline{\fa} = \overline{\fa}.$$ This means that $\overline{\fa}$ is a linked ideal by $(\overline{y_1},..., \overline{y_l})$. Now, using the inductive hypothesis, $\overline{R}$, and so $R,$ is Gorenstein.

\end{proof}

As another consequence of \ref{l13}, one can characterize the regular local rings, too.

\begin{thm}\label{t12}
 A local ring $(R,\fm)$ is regular if and only if there exists a maximal $R$-regular sequence $x_1,...,x_t$ such that $\fm$ is not linked by $(x_1,...,x_t)$.
\end{thm}

\begin{proof}
Let $R$ be a regular local ring and set $t : = \dim R$. Then, there exists an $R$-regular sequence $x_1,...,x_t$ such that $\fm = (x_1,...,x_t)$. Therefore, $\fm$ is not linked by $(x_1,...,x_t)$.

Now, assume that there exists a maximal $R$-regular sequence $x_1,...,x_t$ such that $\fm$ is not linked by $(x_1,...,x_t)$. As $\fm \in \Ass\frac{R}{(x_1,...,x_t)}$, by \ref{l13} (i), $\fm = (x_1,...,x_t)$. Therefore, $R$ is a regular ring.

\end{proof}


\bibliographystyle{amsplain}

\end{document}